\documentclass[12pt, a4paper]{article}
\usepackage{amsfonts,amssymb,amsmath,amscd,latexsym,makeidx,theorem}

\usepackage{hyperref}
\usepackage{color}

\newtheorem{thm}{Theorem}[section]

\newtheorem{lem}[thm]{Lemma}
\newtheorem{prop}[thm]{Proposition}
\newtheorem{cor}[thm]{Corollary}

\newtheorem{rem}{Remark}

\newenvironment{proof}{\noindent\emph{Proof.}}{\hfill$\square$\medskip}

\newcommand{\Ls}{\Lambda_{\mathrm{sph}}}
\newcommand{\D}{\Delta}

\newcommand{\R}{\mathbb{R}}
\newcommand{\ve}{\varepsilon}

\renewcommand{\(}{\left(}
\renewcommand{\)}{\right)}

\newcommand{\la}{\lambda}
\newcommand{\La}{\Lambda}
\title{Normal conformal metrics on $\R^4$ with $Q$-curvature having power-like growth}
\author{Ali Hyder\thanks{Supported by the Swiss National Science Foundation Grant No. P400P2-183866.}\\ {\small Johns Hopkins University}\\ {\small \texttt{ahyder4@jhu.edu}} \and Luca Martinazzi\\  {\small Universit\`a di Padova} \\ {\small \texttt{luca.martinazzi@math.unipd.it}}}

\begin{document}
\maketitle
 
\begin{abstract}
Answering a question by M. Struwe \cite{StrLiou} related to the blow-up behaviour in the Nirenberg problem, we show that the prescribed $Q$-curvature equation
$$\Delta^2 u=(1-|x|^p)e^{4u}\text{ in }\R^4,\quad \Lambda:=\int_{\R^4}(1-|x|^p)e^{4u}dx<\infty$$
has normal solutions (namely solutions which can be written in integral form, and hence satisfy $\Delta u(x) =O(|x|^{-2})$ as $|x|\to \infty$) if and only if $p\in (0,4)$ and
$$\(1+\frac{p}{4}\)8\pi^2\le \Lambda <16\pi^2.$$

We also prove existence and non-existence results for the positive curvature case, namely for $\Delta^2 u=(1+|x|^p)e^{4u}$ in $\R^4$, and discuss some open questions.
\end{abstract}
 
\section{Introduction}
Given the prescribed $Q$-curvature equation on $\R^4$
\begin{equation}\label{eqQ}
\Delta^2 u = K e^{4u}\quad \text{in }\R^4,
\end{equation}
where $K\in L^\infty_{loc}(\R^4)$ is a given function, we say that $u$ is a normal solution to \eqref{eqQ} if $Ke^{4u}\in L^1(\R^4)$ and $u$ solves the integral equation
\begin{equation}\label{eqI}
u(x)=\frac{1}{8\pi^2}\int_{\R^4}\log\(\frac{|y|}{|x-y|}\) K(y)e^{4u(y)}dy +c,
\end{equation}
where $c\in \R$. 
It is well known that \eqref{eqI} implies \eqref{eqQ}, while the converse is not true, see e.g. \cite{CC,Lin}. 

If the right-hand side of \eqref{eqQ} is slightly more integrable, more precisely if $\log(|\cdot|)Ke^{4u}\in L^1(\R^4)$, then \eqref{eqI} is equivalent to 
\begin{equation}\label{eqIbis}
u(x)=\frac{1}{8\pi^2}\int_{\R^4}\log\(\frac{1}{|x-y|}\) K(y)e^{4u(y)}dy +c'.
\end{equation}
We will often use this second version for convenience.

A solution $u$ to \eqref{eqQ} has the geometric property that the conformal metric $e^{2u}|dx|^2$ on $\R^4$ has $Q$-curvature equal to $K$. For this reason equation \eqref{eqQ} has received a lot of attention in the last decades, including lower and higher order analogs, both when $K$ is constant and non-constant, see e.g. \cite{CK,CY, MarClass} and the references therein.

Part of the interest in solutions to \eqref{eqQ} arises from the Nirenberg problem, i.e. the problem of finding whether a given function on a given smooth Riemannian manifold (M, g), usually compact and without boundary, and of dimension $4$ in our case (similar considerations hold in any dimension), can be the Q-curvature of a conformal metric $e^{2u}g$. The variational methods or geometric flows used to study such problems, usually lead to lack of compactness issues, and upon suitable scaling at a blow-up point one often obtain as solution of \eqref{eqQ}. Moreover, because of a priori gradient and volume bounds, usually such solutions are normal, in the sense of \eqref{eqI}, with $Ke^{4u}\in L^1(\R^4)$, and if the prescribed curvature function is always positive and continuous, a blow-up argument will lead to a normal solution to \eqref{eqQ} with $K$ constant. Such solutions have been studied by \cite{Lin} (in other dimensions by \cite{CC, H-odd, H-clas, JMMX, MarClass, WX} and others), and always take the form, when $K=6$,
\begin{equation}\label{uspherical}
u(x)=\log\(\frac{2\lambda}{1+\lambda^2|x-x_0|^2}\),\quad \lambda>0,\quad x_0\in \R^4.
\end{equation}

More recently, though, Borer, Galimberti and Struwe \cite{BGS} studied a \emph{sign-changing} prescribed Gaussian curvature problem in dimension $2$, and, under the generic assumption that the prescribed curvature has only non-degenerate maxima, their blow-up analysis led possibly to either normal solutions to \eqref{eqQ} with $K>0$ constant, or to normal solutions to
\begin{equation}\label{eqG}
-\Delta u(z)=\(1+ A(x,x)\)e^{2u(x)},\quad \text{in }\R^2,
\end{equation}
where $A$ is a negative definite bilinear map. Later Struwe \cite{StrJEMS} showed that in fact \eqref{eqG} admits no \emph{normal} solutions.

A similar analysis was done by Galimberti \cite{Gal} and Ng\^o-Zhang \cite{NZ} in dimension $4$, which led to normal solution to \eqref{eqQ} with $K(x)=(1+A(x,x))$, again with $A(x,x)$ a negative definite bilinear form, assuming that the prescribed curvature $f$ has non-degenerate maxima. In this case, it was not possible to use the ideas of \cite{StrJEMS} to rule out the existence of such normal solutions. On the other hand in case the prescribed curvature function $f$ has derivatives vanishing up to order $3$, and $4$-th order derivative negative definite, blow-up leads to normal solutions to \eqref{eqQ} with $K(x)=(1+A(x,x,x,x))$ and $A$ a negative definite symmetric $4$-linear map, and in this case Struwe \cite{StrLiou} has recently proven that such solutions do not exist.

Then the non-degenerate case remained open, namely whether solutions to \eqref{eqQ} with $K(x)=(1+A(x,x))$,  $A$ bilinear and negative definite, do exist. In this paper we answer this question in the affirmative. In fact, focusing on the case $K(x)=(1-|x|^p)$ for any $p\in (0,4)$, we shall see that \eqref{eqQ} has a normal solution with prescribed total curvature $\Lambda$, if and only if $\Lambda$ lies in a certain range.

More precisely, for {$p>0$} we set
\begin{equation}\label{defLambda*}
\Ls:=16\pi^2,\quad \Lambda_{*,p}:=(4+p)2\pi^2=\(1+\frac{p}{4}\)\frac{\Ls}{2}. 
\end{equation}
The constant $\Ls= 6|S^4|$ is the total $Q$-curvature of  the sphere $S^4$.

We start with a non-existence result. 

\begin{thm}\label{thm0} Fix $p>0$. For any  $\Lambda\in ({-\infty},\Lambda_{*,p})\cup[\Ls,\infty)$ the problem
\begin{align}\label{eq-0}
\D^2 u=(1-|x|^p)e^{4u}\quad\text{in }\R^4,\quad \Lambda=\int_{\R^4}(1-|x|^p)e^{4u} dx,
\end{align}
admits no normal solutions. In particular, for $p\ge 4$, Problem \eqref{eq-0} admits no normal solutions.
\end{thm}

Theorem \ref{thm0} is based on a Pohozaev identity (Proposition \ref{poho-3}) and some asymptotic estimates at infinity.

\medskip

Based on a variational approach of A. Chang and W. Cheng \cite{CC}, together with a blow-up argument, we shall then prove the following existence result.

\begin{thm}\label{thm1} Let $p\in (0,4)$ be fixed. Then for every $\Lambda \in (\Lambda_{*,p},\Ls)$ there exists a (radially symmetric) \emph{normal} solution to Problem \eqref{eq-0}.
Such solutions (in fact, every normal solution to \eqref{eq-0}) have the asymptotic behavior
\begin{equation}\label{asymp}
u(x)= -\frac{\Lambda}{8\pi^2}\log|x| + C +O(|x|^{-\alpha}), \quad\text{as }|x|\to\infty,
\end{equation}
for every $\alpha\in [0,1]$ such that $\alpha<\frac{\Lambda-\Lambda_{*,p}}{2\pi^2}$, and
\begin{equation}\label{asymp3}
|\nabla^\ell u(x)|=O(|x|^{-\ell}), \quad \text{as }|x|\to\infty, \quad \ell=1,2,3.
\end{equation}
\end{thm}

Theorems \ref{thm0} and \ref{thm1} leave open the case $\Lambda =\Lambda_{*,p}$, which is borderline from the point of view of integrability, in the sense that \eqref{asymp} is compatible with the integrability condition in $\eqref{eq-0}$ if $\Lambda>\Lambda_{*,p}$, but for $\Lambda=\Lambda_{*,p}$, \eqref{asymp} degenerates to
$$-\frac{\Lambda+o(1)}{8\pi^2}\log|x| +O(1)\le u(x)\le  -\frac{\Lambda}{8\pi^2}\log|x| + O(1), \quad\text{as }|x|\to\infty,$$
(see Lemma \ref{lemasymp} and Lemma \ref{lemasymp2}), which is not incompatible the integrability of $(1-|x|)e^{4u}$.

We shall study the case $\Lambda=\Lambda_{*,p}$ from the point of view of compactness: while solutions of Theorem \ref{thm1} must necessarily blow up as $\Lambda\uparrow \Ls$ (see Theorem \ref{thm2}), we find that such solutions remain compact as $\Lambda\downarrow\Lambda_{*,p}$.

\begin{thm}\label{thm1b}
Fix $p\in (0,4)$. Given any sequence $(u_k)$ of radial normal solutions to \eqref{eq-0} with $\Lambda=\Lambda_k\in [\Lambda_{*,p}, \Ls)$ and $\Lambda_k\to \bar\Lambda \in [\Lambda_{*,p}, \Ls)$, up to a subsequence $u_k\to \bar u$ locally uniformly, where $\bar u$ is a normal (and radial) solution to \eqref{eq-0} with $\Lambda=\bar\Lambda$.

In particular, choosing $\Lambda_k\downarrow \Lambda_{*,p}$ and $u_k$ given by Theorem \ref{thm1}, we obtain that \eqref{eq-0} has a normal solution $u$ also for $\Lambda=\Lambda_{*,p}$.
Moreover such $u$ satisfies
\begin{equation}\label{asymp2}
u(x) {\le}-\frac{\Lambda_{*,p}}{8\pi^2}\log|x| -\(\frac12+o(1)\)\log\log |x|, \quad \text{as }|x|\to\infty,
\end{equation}
and \eqref{asymp3}.
\end{thm}

\noindent \textbf{Open problem} The solutions given by Theorems \ref{thm1} and \ref{thm1b} are radially symmetric by construction. It is open whether all normal solutions are radially symmetric (compare to \cite{Lin}), and whether they are unique, for every given $\Lambda\in [\Lambda_{*,p},\Ls)$.  

\medskip

\noindent\textbf{Open problem} Is the inequality in \eqref{asymp2} actually an equality (see Lemma \ref{lem-4.6} for a sharper version of \eqref{asymp2})?

\medskip

The proof of  Theorem \ref{thm1b} relies on blow-up analysis and quantization, as studied in \cite{Rob}, \cite{MarOpen}, which implies that in case of loss of compactness the total $Q$-curvature converges to $\Ls$, which is a contradiction if $\Lambda_k \to \bar \Lambda\in [\Lambda_{*,p},\Ls)$. An important part of this argument is to rule out loss of curvature at infinity, see Lemma \ref{lemLambdabar}.

On the other hand, as $\Lambda\uparrow \Ls$, the non-existence result of Theorem \ref{thm0} leaves open only two possibilities: loss of curvature at infinity, or loss of compactness. In the next theorem we show that the second case occurs.



\begin{thm}\label{thm2} Fix $p\in (0,4)$ and let $(u_k)$ be a sequence of radial normal solutions of \eqref{eq-0} with $\Lambda=\Lambda_k\uparrow \Ls$ (compare to Theorem \ref{thm1}) as $k\to\infty$. Then
\begin{equation}\label{concentration}
(1-|x|^p)e^{4u_k}\rightharpoonup \Ls \delta_0 \quad \text{as }k\to\infty,
\end{equation}
weakly in the sense of measures, and setting
$$\eta_k(x):= u_k(r_k x) -u_k(0)+\log 2,\quad r_k:=12 e^{-u_k(0)},$$ we have
\begin{equation}\label{convetak}
\eta_k(x)\to \log\(\frac{2}{1+|x|^2}\)\quad \text{locally uniformly in } \R^4.
\end{equation}
\end{thm}

Having addressed the case $K(x)=1-|x|^p$, we now analize the case $K(x)=1+|x|^p$.

Similar to Theorems \ref{thm0}, \ref{thm1} and \ref{thm1b}, one can ask for which values of $\Lambda$ problem 
\begin{align}\label{eq-positive}
\D^2u=(1+|x|^p)e^{4u}\quad\text{in }\R^4,\quad  \Lambda:=\int_{\R^4} (1+|x|^p)e^{4u}dx<\infty
\end{align}
admits a normal solution. The following result gives a complete answer for $p\in (0,4]$ and a partial answer for $p>4$. Let $\Lambda_{*,p}$ be as in \eqref{defLambda*}.

\begin{thm}\label{thm-positive2}
For $p\in (0,4]$, Problem \eqref{eq-positive} has a normal solution if and only if 
\begin{align} \label{nece0}
\Ls <\Lambda<2\Lambda_{*,p}.
\end{align}
For $p>4$
\begin{align} \label{nece}
\Lambda_{*,p} <\Lambda<2\Lambda_{*,p}
\end{align}
is a necessary condition for the existence of normal solutions to \eqref{eq-positive}, and there exists $\ve_p>0$ such that
\begin{align} \label{nece2}
\Lambda_{*,p}+\ve_p <\Lambda<2\Lambda_{*,p}
\end{align}
is a necessary condition for the existence of \emph{radial} normal solutions to \eqref{eq-positive}. Finally, for $p>4$ and for every
\begin{align}\label{suffi}
\frac{p\Ls}4 <\Lambda<2\Lambda_{*,p}
\end{align}
there exists a radially symmetric normal solution to \eqref{eq-positive}.
\end{thm}

While the necessary condition \eqref{nece0}-\eqref{nece} follow from the Pohozaev identity, the existence part and the more restrictive condition \eqref{nece2} are based on blow-up analysis. To study blow-up at the origin we use again the methods of \cite{Rob} and \cite{MarOpen}, and to avoid vanishing of curvature at infinity, which can be seen as a form of blow-up at infinity or, equivalently, as blow-up of the Kelvin transform at the origin, we will use the blow-up analysis and classification result of \cite{HMM} for normal solutions of \eqref{eqQ} with $K(x)=|x|^p$.

\medskip

\noindent\textbf{Open problems} In the case $p>4$ it is not know whether the condition \eqref{suffi} is also necessary. This is also related to the problem of uniqueness/multiplicity  of solutions to \eqref{eq-positive} for a given $\Lambda$ (open also in the case $p\in (0,4]$), and to the problem of the existence of a minimal value of $\Lambda$ for which \eqref{eq-positive} admits a solution, in analogy with Theorem \ref{thm1b}.

\medskip

\noindent\textbf{Open problem} Every \emph{radial} solution to \eqref{eqQ} with $K(x)=1- |x|^p$, $p>0$ must have finite total curvature, namely $Ke^{4u}\in L^1(\R^4)$, see Proposition \ref{propfinite}. The same happens when $K=const>0$, but in this case Albesino \cite{Alb} has recently proven the existence of non-radial solutions $u$ with $Ke^{4u}\not\in L^1(\R^4)$. It would be interesting to see whether there exist non-radial solutions to \eqref{eqQ} with infinite total $Q$-curvature also in the case $K(x)=1-|x|^p$.

\medskip

\noindent\textbf{Open problem} Inspired by \cite{HD, HVol,MarVol}, can one find (non-normal) solutions to \eqref{eqQ} with $K=(1-|x|^p)$ and arbitrarily large but finite total $Q$-curvature $\Lambda=\int_{\R^4}Ke^{4u}dx$? In the case of $K=(1+|x|^p)$, using the methods from \cite{HMM} it should be possible to prove the upper bound $\Lambda<2\Lambda_{*,p}$ for (not necessarily normal) \emph{radial} solutions.

\section{Some preliminary results and proof of Theorem \ref{thm0}}

We start with a Pohozaev-type identity that will be used several times.

\begin{prop}[Pohozaev identity]\label{poho-3}
{Let $K(x)=(1\pm |x|^p)$} and let $u$ be a solution to the integral equation 
\begin{equation*}
u(x)= \frac{1}{8\pi^2}\int_{\R^4}  \log\left(\frac{|y|}{|x-y|}\right)K(y) e^{4u(y)} dy + c 
\end{equation*}
for some $c\in \R$, with $Ke^{4 u}\in L^1(\R^4)$ and $ |\cdot|^p e^{4 u} \in L^1(\R^4)$. If 
\begin{align}\label{cond-poho}
\lim_{R\to\infty} R^{4+p} \max_{|x|=R} e^{4u(x)}=0,
\end{align}
then, denoting $\Lambda:=\int_{\R^4}(1\pm|x|^p)e^{4u(x)} dx,$ we have
\begin{equation}\label{general Poho}
\frac{\Lambda}{\Ls}\left(\Lambda-\Ls\right) =\pm \frac{p}{4}\int_{\R^4}|x|^p e^{4u(x)}dx.
\end{equation}
\end{prop}

\begin{proof} Following the proof of Proposition A.1 in \cite{HMM}, we need to show that as $R\to\infty $ $$R\int_{\partial B_R}|x|^pe^{4u}d\sigma\to0\quad\text{and }\int_{|x|\leq R}\int_{|y|\geq R}\frac{|x+y|}{|x-y|} |x|^pe^{4u(x)}|y|^pe^{4u(y)}dydx\to0.$$
By \eqref{cond-poho} the boundary term goes to $0$ as $R\to\infty$. For the double integral term we divide the domain of $B_R^c$ into   $B_{2R}^c$ and $B_{2R}\setminus B_R$. Now using $|\cdot|^pe^{4u}\in L^1(\R^4)$ and \eqref{cond-poho} we estimate 
\begin{align*}
\int_{|x|\leq R}\int_{R\leq|y|\leq 2R}\frac{|x+y|}{|x-y|} |x|^pe^{4u(x)}|y|^pe^{4u(y)}dydx &=  o\(\frac{1}{R^3}\int_{|x|\leq R}  |x|^pe^{4u(x)} \int_{R\leq |y| \leq2 R}\frac{ dy}{|x-y|} dx \)\\
 &=o(1),\quad \text{as }R\to\infty
\end{align*}
and
\begin{align*} \int_{|x|\leq R}\int_{|y|\geq 2R}\frac{|x+y|}{|x-y|} |x|^pe^{4u(x)}|y|^pe^{4u(y)}dydx &\leq  C \int_{|x|\leq R}  |x|^pe^{4u(x)}  dx \int_{ |y| \geq2 R} |y |^p e^{4u(y)}dy \\ & =o(1),\quad \text{as }R\to\infty.
\end{align*}
\end{proof}

Another basic tool often used is the Kelvin transform.

\begin{prop}\label{PKelvin} Let $u$ be a normal solution to \eqref{eqQ} with $K\in L^\infty_{loc}(\R^4)$ and $Ke^{4u}\in L^1(\R^4)$. Then the function 
\begin{equation}\label{defKelvin}
\tilde u(x)= u\(\frac{x}{|x|^2}\)-\alpha \log|x|,\quad \text{for }x\ne 0,\quad \alpha:=\frac{1}{8\pi^2}\int_{\R^4}Ke^{2u}dx,
\end{equation}
satisfies
$$\tilde u(x)=\frac{1}{8\pi^2}\int_{\R^4}\log\(\frac{1}{|x-y|}\)K\(\frac{y}{|y|^2}\)\frac{e^{4\tilde u(y)}}{|y|^{8-4\alpha}} dy +c,$$
namely $\tilde u$ is a normal solution to
$$\Delta^2\tilde u(x) =K\(\frac{x}{|x|^2}\)\frac{e^{4\tilde u}}{|x|^{8-4\alpha}}.$$
\end{prop}

\begin{proof} Starting from \eqref{eqI}, with a change of variables, and using that $|x||y|\left|\frac{x}{|x|^2}-\frac{y}{|y|^2}\right|=|x-y|$, we obtain
\begin{equation}\label{Kelvin}
\begin{split}
\tilde u(x)&=\frac{1}{8\pi^2}\int_{\R^4}\log\(\frac{|y|}{|x|\left|\frac{x}{|x|^2}-y\right|}\)K\(y\)e^{4 u(y)}dy+c\\
&=\frac{1}{8\pi^2}\int_{\R^4}\log\(\frac{1}{|x||y|\left|\frac{x}{|x|^2}-\frac{y}{|y|^2}\right|}\)K\(\frac{y}{|y|^2}\)\frac{e^{4\tilde u(y)}}{|y|^{8-4\alpha}}dy+c\\
&=\frac{1}{8\pi^2}\int_{\R^4}\log\(\frac{1}{|x-y|}\)K\(\frac{y}{|y|^2}\)\frac{e^{4\tilde u(y)}}{|y|^{8-4\alpha}}dy+c.
\end{split}
\end{equation}
\end{proof}

We now start studying the asymptotic behavior of normal solutions to \eqref{eqQ} under various assumptions.

\begin{lem}\label{lemasymp} Let $u$ solve the integral equation
\begin{equation}\label{eqI2}
u(x)=\frac{1}{8\pi^2}\int_{\R^4}\log\(\frac{|y|}{|x-y|}\) K(y)e^{4u(y)}dy +c,
\end{equation}
where $K(y)\le 0$ for $|y|\ge R_0$, for a given $R_0\ge 0$, and $Ke^{4u}\in L^1$. Then we have
\begin{equation}\label{asymp4}
u(x)\le -\frac{\Lambda}{8\pi^2}\log|x| +O(1),\quad \text{as }|x|\to\infty,
\end{equation}
where
$$\Lambda=\int_{\R^4} K(y)e^{4u(y)}dy\in \R  .$$
\end{lem} 

\begin{proof} Choose $x$ such that $|x|\ge 2R_0$. Without loss of generality we can assume $R_0\ge 2$. Split $\R^4= A_1\cup A_2\cup A_3$ where
$$A_1=B_{\frac{|x|}{2}}(x),\quad A_2= B_{R_0}(0), \quad A_3=\R^4\setminus(A_1\cup A_2).$$
Using that
$$\log\(\frac{|y|}{|x-y|}\)\ge 0, \quad K(y)\le 0\quad \text{in }A_1,$$
we get
$$\int_{A_1}\log\(\frac{|y|}{|x-y|}\) K(y)e^{4u(y)}dy \le 0.$$
For $y\in A_2$ we have $\log\(\frac{|y|}{|x-y|}\)=-\log|x|+O(1)$ as $|x|\to\infty$, so that
\begin{align*}
\int_{A_2}\log\(\frac{|y|}{|x-y|}\) K(y)e^{4u(y)}dy& =(-\log|x|+O(1))\int_{A_2}Ke^{4u}dy\\
&= -\log|x|\int_{A_2}Ke^{4u}dy+O(1).
\end{align*}
For $|y|\ge R_0\ge 2$ and $|x-y|>\frac{|x|}{2}$ we have $|x-y|\le |x|+|y|\le |x||y|$ so that
$$\int_{A_3}\log\(\frac{|y|}{|x-y|}\) K(y)e^{4u(y)}dy \le - \log |x| \int_{A_3} K(y)e^{4u(y)}dy.$$
Summing up
\begin{align*}
u(x)&\le -\frac{1}{8\pi^2} \log|x|\int_{A_2\cup A_3} K(y)e^{4u(y)}dy+O(1)\\
&\le -\frac{\Lambda}{8\pi^2} \log|x|+O(1),
\end{align*}
where again we used that $K\le 0$ in $A_1$.
\end{proof}

\begin{cor}\label{cornonex} Given $p\in (0,4)$ there is no normal solution to \eqref{eq-0} for $\Lambda\ge \Ls=16\pi^2$.
\end{cor}

\begin{proof} Assume that $u$ solves \eqref{eq-0} for some $\Lambda\ge \Ls$. Then, by Lemma \ref{lemasymp}, $u$ satisfies \eqref{asymp4}, which implies that assumption \eqref{cond-poho} in Proposition \ref{poho-3} is satisfied. Then \eqref{general Poho} implies $\Lambda <\Ls$, a contradiction.
\end{proof}

\begin{lem}\label{lemasymp2} Given $p>0$ let $u$ be a normal solution to \eqref{eq-0} for some $\Lambda\in \R$. Then $\Lambda\ge \Lambda_{*,p}$ and
\begin{equation}\label{asymp5}
u(x)= -\frac{\Lambda +o(1)}{8\pi^2}\log|x|\quad \text{as }|x|\to \infty.
\end{equation}
\end{lem}

\begin{proof}
We start by proving \eqref{asymp5}. We write $u=u_1+u_2$, where
$$u_2(x)=-\frac{1}{8\pi^2}\int_{B_1(x)}\log\frac{1}{|x-y|}|y|^p e^{4u(y)}dy.$$
Then we have
$$u_1(x)=  -\frac{\Lambda +o(1)}{8\pi^2}\log|x|\quad \text{as }|x|\to \infty.$$
We now claim that $\frac{\Lambda}{2\pi^2}>p$. Then we have that $|y|^pe^{4u_1(y)}\leq C$ on $\R^4$. This, and as $u_2\leq 0$, we easily get that  $$|u_2(x)|\leq C\int_{B_1(x)}\log\frac{1}{|x-y|}dy\leq C,$$
hence \eqref{asymp5} is proven. 

In order to prove the claim, given $R\gg 1$ and $|x|\geq R+1$ we write  $$-u_2(x)= \int_{B_R^c}h(R)\log\frac{1}{|x-y|}\chi_{|x-y|\leq 1}d\mu(y),\quad d\mu(y)=\frac{|y|^pe^{4u}}{\int_{B_R^c}|y|^pe^{4u}dy} dy,$$ where $$h(R)=\frac{1}{8\pi^2}\int_{B_R^c}|y|^pe^{4u}dy=o_R(1)\xrightarrow{R\to\infty}0.$$
By Jensen's inequality and Fubini's theorem we obtain
\begin{align*} \int_{R+1<|x|<2R}e^{-4u_2}dx\leq \int_{B_R^c}   \int_{R+1<|x|<2R} \left(1+\frac{1}{|x-y|^{4h(R)}}\right) dx d\mu(y)\leq CR^4. \end{align*}
Therefore, by H\"older inequality $$R^4\approx  \int_{R+1<|x|<2R}e^{2u_2}e^{-2u_2}dx\leq CR^2\left( \int_{R+1<|x|<2R}e^{4u_2}dx\right)^\frac12.$$ If $\frac{\Lambda}{2\pi^2}\leq p$, then we have that $|y|^pe^{4u_1(y)}\geq \frac{1}{|y|} $ for $|y|$ large. 
Hence, $$o_R(1)= \int_{R+1<|x|<2R}|x|^pe^{4u_1}e^{4u_2}dx\gtrsim \frac1R  \int_{R+1<|x|<2R}  e^{4u_2}dx,$$ a contradiction.

Now that \eqref{asymp5} is proven, we have that $\Lambda< \Lambda_{*,p}$ contradicts $(1-|x|^p)e^{4u}\in L^1(\R^4)$, hence we must have $\Lambda\ge \Lambda_{*,p}$.
\end{proof}

\begin{lem}\label{lem-gen}
Let $w\in C^0(B_1\setminus \{0\})$ be given by
$$w(x)=\int_{B_1}\log\(\frac{1}{|x-y|}\) f(y)dy,$$
for some nonnegative $f\in L^1(B_1)$.   If $w(x_k)=O(1)$ for some $x_k\to 0$ then
$$\int_{B_1}\log\(\frac{1}{|y|}\)f(y)dy<\infty.$$
\end{lem}

\begin{proof}
Let $x_k\to 0$ be such that $w(x_k)=O(1)$ as $k\to\infty$. Then we have
\begin{align*}
O(1)=\int_{B_1} \log\(\frac{1}{|x_k-y|}\) f(y)dy&\geq O(1)+\int_{2|x_k|\leq |y|\leq 1}\log\(\frac{1}{|x_k-y|}\)f(y)dy  \\&=O(1)+\int_{2|x_k|\leq |y|\leq 1}\log\(\frac{1}{|y|}\) f(y)dy.
\end{align*}
The lemma follows by taking $k\to\infty$. 
\end{proof}

\begin{lem}\label{lemmatildeu} Let $u$ be a normal solution to \eqref{eq-0} with $\Lambda =\Lambda_{*,p}$, and let $\tilde u$ as in \eqref{defKelvin} be its Kelvin transform, namely  
\begin{equation}\label{defutilde}
\tilde u(x)=u\(\frac{x}{|x|^2}\)-\(1+\frac p4\)\log|x|,\quad x\neq 0.
\end{equation}
Then
\begin{align}\label{lim-0}
\lim_{x\to 0}\tilde u(x)=-\infty, \\
\lim_{x\to 0}\Delta \tilde u(x)= +\infty. \label{lim-1}
\end{align}
\end{lem}

\begin{proof}
Observe that $\sup_{B_1}\tilde u<\infty$ by Lemma \ref{lemasymp}. 
To prove \eqref{lim-0} we assume by contradiction that $\tilde u(x_k)=O(1)$ for a sequence $x_k\to 0$.
By Proposition \ref{PKelvin} we have
\begin{equation}\label{Kelvin2}
\tilde u(x)=\frac{1}{8\pi^2}\int_{\R^4}{\log\(\frac{1}{|x-y|}\)}\(1-\frac{1}{|y|^p}\)\frac{e^{4\tilde u(y)}}{|y|^{4-p}}dy+c,\quad x\ne 0.
\end{equation}
Since $\tilde u\leq C$ in $B_1$,  from \eqref{defutilde} and the continuity of $u$ it follows that
\begin{equation}\label{35}
\(1+\frac{1}{|y|^p}\)\frac{e^{4\tilde u(y)}}{|y|^{4-p}}\leq \frac{C}{|y|^8}\quad \text{in } B_1^c.
\end{equation}
Then from \eqref{Kelvin2} we obtain
\begin{align}\label{36}
\tilde u(x)= -\frac{1}{8\pi^2}\int_{B_1}\log\left(\frac{1}{|x-y|}\right)\frac{e^{4\tilde u(y)}}{|y|^{4}}dy+O(1),\quad 0<|x|<1.
\end{align}
Then by Lemma \ref{lem-gen} applied to \eqref{36}, and with a change of variables, we get
$$\int_{B_1^c}\log(|y|)|y|^pe^{4u(y)}dy=\int_{B_1}\log\(\frac{1}{|y|}\)\frac{e^{4\tilde u(y)}}{|y|^4}dy<\infty.$$
Then, as $|x|\to\infty$ we obtain
\begin{equation}\label{}
\begin{split}
\int_{|y|\le \sqrt{x}}(1-|y|^p)e^{4u(y)}dy&=\Lambda_{*,p}-\int_{|y|> \sqrt{x}}(1-|y|^p)e^{4u(y)}dy\\
&=\Lambda_{*,p}+O\(\frac{1}{\log(|x|)}\int_{|y|> \sqrt{x}}\log(|y|)(1-|y|^p)e^{4u(y)}dy\)\\
&=\Lambda_{*,p}+O\(\frac{1}{\log(|x|)}\).
\end{split}
\end{equation}
Moreover, $u$ can be given by \eqref{eqIbis} with $K=1-|x|^p$. Hence, for $|x|>>1$
\begin{align*}
u(x)&=O(1)+\frac{1}{8\pi^2}\left(  \int_{|y|\leq \sqrt{|x|}} +\int_{\sqrt{|x|}\leq|y|\leq 2|x|}+\int_{|y|\geq 2|x|} \right) \log\(\frac{1}{|x-y|}\) (1-|y|^p)e^{4 u(y)}dy \\
&\geq O(1)-\frac{\Lambda_{*,p}}{8\pi^2}\(1+O\(\frac{1}{\log|x|}\)\)\log|x|\\
& =-\frac{\Lambda_{*,p}}{8\pi^2}\log|x|+O(1),
\end{align*}
a contradiction to $|\cdot|^pe^{4u}\in L^1(\R^4)$, which completes the proof of \eqref{lim-0}.

To prove \eqref{lim-1} we differentiate into \eqref{Kelvin2} and obtain
$$\Delta \tilde u(x)=\frac{1}{2\pi^2} \int_{\R^4}\frac{1}{|x-y|^2}\(1-\frac{1}{|y|^p}\)\frac{e^{4\tilde u(y)}}{|y|^{4-p}}dy,$$
and as before, we use \eqref{35} to get
\begin{align*}
\Delta \tilde u(x)&=\frac{1}{2\pi^2} \int_{B_1}\frac{1}{|x-y|^2}\frac{e^{4\tilde u(y)}}{|y|^4}dy+O(1)\\
&\ge C \int_{B_1}\log\(\frac{1}{|x-y|}\)\frac{e^{4\tilde u(y)}}{|y|^4}dy +O(1)\\
&= -C8\pi^2 \tilde u(x)+O(1)  \to \infty  \quad \text{as }x\to 0.
\end{align*}
\end{proof}

\begin{prop}\label{propnonex} There exists no normal solution to \eqref{eq-0} for $p\geq4$.
 \end{prop} 
\begin{proof} Assume by contradiction that there exists a normal solution $u$ to \eqref{eq-0} for some $p\geq 4$. Then necessarily we have that $\Lambda\geq\Lambda_{*,p}$, thanks to \eqref{asymp5}. Now we distinguish the following two cases. 

\noindent\textbf{Case 1} $\Lambda>\Lambda_{*,p}$. 

Since $u\leq -\frac{\Lambda}{8\pi^2}\log|x|+C$ for $|x|\geq 1$ by Lemma \ref{lemasymp}, we see that $u$ satisfies \eqref{cond-poho}. Hence, by \eqref{general Poho}, $$\Ls\leq \Lambda_{*,p}<\Lambda<\Ls,$$ a contradiction. 

\noindent\textbf{Case 2}  $\Lambda=\Lambda_{*,p}$. By Lemma \ref{lemmatildeu} we see that \eqref{cond-poho} is satisfied, and we arrive at a contradiction as in Case 1.
\end{proof}

\noindent\emph{Proof of Theorem \ref{thm0}} Combine Corollary \ref{cornonex}, Lemma \ref{lemasymp2} and Proposition \ref{propnonex}.
\endproof

For $\Lambda=\Lambda_{*,p}$ we obtain a sharper version of \eqref{lim-0} if $u$ is radial.

\begin{lem}\label{lem-4.6}
Given $p\in (0,4)$, let $u$ be a radially symmetric normal solution to \eqref{eq-0} with $\Lambda=\Lambda_{*,p}$. Then
$$\limsup_{|x|\to\infty}\frac{u(x)+(1+\frac p4)\log |x|}{\log\log|x|}=-\frac{1}{2}$$
\end{lem}
\begin{proof}   
We set $\tilde u$ as in \eqref{defutilde}, so that it satisfies \eqref{Kelvin2}.
By Lemma \ref{lemmatildeu} we get 
$$\lim_{r\to0}\tilde u(r)=-\infty,\quad \lim_{r\to0}\D\tilde  u(r)=+\infty.$$
In particular, $\tilde u$ is monotone increasing in a small neighborhood of the origin. Using this and \eqref{36} we estimate for $|x|\to 0$
\begin{align*} -\tilde u(x)&\geq \frac{1}{8\pi^2}\int_{2|x|\leq |y|<1}\log\left(\frac{1}{|x-y|}\right)\frac{e^{4\tilde u(y)}}{|y|^{4}}dy +O(1)\\
 & =  \frac{1}{8\pi^2}\int_{2|x|\leq |y|<1}\log\left(\frac{1}{|y|}\right)\frac{e^{4\tilde u(y)}}{|y|^{4}}dy +O(1)  \\ 
& \geq  \frac{e^{4\tilde u(x)}}{8\pi^2}\int_{2|x|\leq |y|<1}\log\(\frac{1}{|y|}\)\frac{ dy}{|y|^{4}} +O(1)  \\  &=\frac{e^{4\tilde u(x)}}{8\pi^2}|S^3|\int_{2|x|}^1\frac{\log\frac1t}{t}dt +O(1).
\end{align*}
{Computing the integral and considering that $|S^3|=2\pi^2$ we get
$$-\tilde u(x)+O(1)\ge \frac{e^{4\tilde u(x)}}{8} (\log(2|x|))^2,\quad \text{as }x\to 0.$$
Taking the logarithm and rearranging we finally get}
$$\limsup_{x \to0}\frac{\tilde u(x)}{\log\log\(\frac{1}{|x|}\)}\leq -\frac12.$$ Next we show that the above  limsup  is actually $-\frac12$. 

We assume by contradiction that the above $\limsup$ is less than $-\frac12$. Then there exists $\ve>0$ such that  for $|x|$ small we have $$\tilde u(x)\leq -\(\frac12+\frac\ve4\)\log\log\frac {1}{|x|}.$$ Hence, from \eqref{36} we obtain for $|x|$ small \begin{align*} -\tilde u(x)&\leq  C\int_{B_1}\log\left(\frac{1}{|x-y|}\right)\frac{dy}{|y|^4|\log|y||^{2+\ve}}+O(1) \\ &=C(I_1+I_2+I_3)+O(1), \end{align*} where $$I_i=\int_{A_i}\log\frac{1}{|x-y|}\frac{dy}{|y|^4|\log|y||^{2+\ve}},$$ $$A_1=B_{\frac{|x|}{2}}, \quad A_2=B_{2|x|}\setminus B_\frac{|x|}{2},\quad A_3=B_1\setminus B_{2|x|}.$$ One easily gets $$I_1\leq \frac{C}{|\log|x||^\ve},\quad I_2\leq \frac{C}{|\log|x||^{1+\ve}},\quad I_3\leq C,$$ a contradiction to $\tilde u(x)\to-\infty$ as $|x|\to0.$
\end{proof}

\begin{prop}\label{propfinite} Let $u$ be a radial solution to
\begin{equation}\label{eqpm}
\D^2 u=(1-|x|^p)e^{4u}\quad\text{in }\R^4,
\end{equation}
for some $p>0$.
Then
\begin{equation}\label{intfinite}
\int_{\R^4}(1+|x|^p)e^{4u}dx<\infty.
\end{equation}
\end{prop}

\begin{proof} If \eqref{intfinite} is false then there exists $R>0$ such that
$$\int_{B_R}(1-|x|^p)e^{4u}dx<0.$$ 
In particular, $(\D u)'(r)<0$ for $r\geq R$. We can consider the following two cases:\\

\noindent \textbf{Case 1:} $\lim_{r\to\infty}\D u(r)\geq 0$.  Then as $\D u$ is monotone decreasing on $(R,\infty)$, we see that $\D u>0$ on $(R,\infty)$. Therefore, by \eqref{identity} we get that $u\geq -C$ in $\R^4$. This, \eqref{eqpm} and \eqref{identity} imply that $\D u(r)\lesssim -r^{p+2}$ as $r\to\infty$, a contradiction. \\

\noindent \textbf{Case 2:}  $\lim_{r\to\infty}\D u(r)<0$. 
In this case, by \eqref{identity}, we have that $u(r)\lesssim -r^2$ as $r\to\infty$, a contradiction to the assumption $(1+|x|^p)e^{4u}\not\in L^1(\R^4)$.
\end{proof}

\section{Proof of Theorem \ref{thm1}}

\subsection{Existence}

The existence part in Theorem \ref{thm1} will be based on the following result, which will be proven in Section \ref{Sec:Prop2.1} using methods from \cite{CC} and \cite{Lin}.
 
\begin{prop}\label{prop1} For  every $0<\Lambda<\Ls$ and  $\lambda>0$, there exists a radial solution $u_\lambda$ to
 \begin{align}
& \D^2 u_\lambda=(\lambda-|x|^p)e^{-|x|^2}e^{4u_\lambda}\quad\text{in }\R^4,\label{eq-1}\\
 &\int_{\R^4}(\lambda-|x|^p)e^{-|x|^2}e^{4u_\lambda}dx=\Lambda,\label{eqLambda1}
\end{align} which is normal, namely $u_\lambda$ solves the integral equation
\begin{equation}\label{prop1eq1}
u_\lambda(x)=\frac{1}{8\pi^2}\int_{\R^4}\log\(\frac{1}{|x-y|}\)\(\lambda-|y|^p\)e^{-|y|^2}e^{4u_\lambda(y)}dy+c_\lambda,
\end{equation}
for some constant $c_\lambda\in \R$.
\end{prop}

We will often use the identity
\begin{align}\label{identity} w(R)-w(r)=\int_{r}^R\frac{1}{\omega_3t^3}\int_{B_t}\D w dxdt,\quad 0\leq r< R,\, w\in C^2 _{rad}(\R^4), \quad \omega_3=2\pi^2,\end{align}
which follows at once from the divergence theorem and the fundamental theorem of calculus.

Let $u_\lambda$ be given as in Proposition \ref{prop1}.
 
\begin{lem}\label{lem1} For every $\lambda>0$ we have $u_\lambda(x)\downarrow -\infty$  and $\Delta u_\lambda(x)\uparrow 0$ as $|x|\to \infty$.
\end{lem} 

\begin{proof} The function
$$r\mapsto \int_{B_r}\Delta^2 u_\lambda(x) dx=\int_{B_r} (\lambda-|x|^p)e^{-|x|^2}e^{4u_\lambda(x)}dx$$
is increasing on $[0,\lambda^\frac 1p]$, and decreasing to $\Lambda$ on $[\lambda^\frac 1p,\infty)$. In particular it is positive for every $r>0$. Then, by \eqref{identity} (applied with $w=\Delta u_\lambda$) we infer that $\Delta u_\lambda(x)$ is an increasing function of $|x|$.

Differentiating under the integral sign from \eqref{prop1eq1} we obtain $$|\D u_\lambda(x)|\leq C\int_{\R^4}\frac{1}{|x-y|^2}(1+|y|^p)e^{-|y|^2}dy\xrightarrow{|x|\to\infty}0,$$ where in the first inequality we have used that $u_\lambda\leq C$ on $\R^4$, thanks to Lemma \ref{lemasymp}. 

This in turn implies that $\Delta u_\lambda <0$, hence $u_\lambda$ is decreasing by \eqref{identity}. Finally $u_\lambda \to -\infty$ as $|x|\to\infty$ follows from  Lemma \ref{lemasymp}.
\end{proof}

\begin{lem}\label{lem2} We have $\lambda e^{4u_\lambda (0)}\to\infty$ as $\la\downarrow 0$.   \end{lem}
\begin{proof} Assume by contradiction that $\la e^{4u_\la(0)}\leq C$ as $\la\to0$. Then $$\Lambda=\int_{\R^4}(\lambda-|x|^p)e^{-|x|^2}e^{4u_\lambda}dx\leq \int_{B_{\la^\frac1p}} \la e^{-|x|^2}e^{4u_\la (0)}\xrightarrow{\la\to0}0, $$
which is absurd.
\end{proof}

Now we set 
\begin{equation}\label{defetalambda}
\eta_\lambda(x)=u_\lambda(r_\la x)-u_\la(0),\quad \la r_\la^4 e^{4u_\la(0)}:=1.
\end{equation}
Notice that $r_\la\to0$ by Lemma \ref{lem2}. 
By definition and Lemma \ref{lem1} we have that
$$\eta_\la\leq 0= \eta_\la(0),\quad  \D \eta_\la(x)\uparrow 0 \quad\text{as }|x|\to\infty,$$
and by a change of variables in \eqref{prop1eq1} we see that $\eta_\lambda$ is a normal solution to
$$\D^2\eta_\la=\(1-\frac{r_\la^p}{\la}|x|^p\)e^{-r_\la^2|x|^2}e^{4\eta_\la}.$$
With a similar change of variables in \eqref{eqLambda1} we also get
\begin{align}\int_{\R^4}\(1-\frac{r_\la^p}{\la}|x|^p\)e^{-r_\la^2|x|^2}e^{4\eta_\la}dx=\Lambda.\label{relation-4}\end{align}
Since $\eta_\la\leq 0$, we have
$$0<\Lambda <\int_{B_\frac{\lambda^{1/p}}{r_\lambda}} e^{-r_\lambda^2|x|^2}e^{4\eta_\lambda}dx \le \mathrm{meas}\(B_\frac{\lambda^{1/p}}{r_\lambda}\),$$
which implies that
\begin{equation}\label{rlala}
\limsup_{\la\to0}\frac{r_\la^p}{\la}<\infty.
\end{equation}
 \begin{lem}\label{lem3} We have
 \begin{equation}\label{limsupDeta}
 \limsup_{\la\to0}|\D \eta_\la(0)|<\infty.
 \end{equation}
 \end{lem} 
 \begin{proof} Assume by contradiction that
\begin{equation}\label{limsupDeta2}
\limsup_{\la\to0}|\D \eta_\la(0)|=\infty.
\end{equation}
Then, using that $\Delta \eta_\lambda(x)\uparrow 0$ as $|x|\to\infty$ for every $\lambda>0$, there exists $R_\la>0$ such that $\D \eta_{\la}(R_\la)=-1$. Then, as $\D\eta_\la\leq -1$ on $[0,R_\la]$, we have that $$\eta_{\la}(r)\leq -\frac {1}{8}r^2\quad \text{for } 0\leq r\leq R_\la.$$
From this, and using \eqref{identity} and \eqref{rlala} one obtains
 \begin{equation*}\begin{split}
 \D\eta_\la(R_\la)-\D \eta_\la(0)&=  O\(\int_0^{R_\la}\frac{1}{t^3}\int_{B_t}\Delta^2 \eta_\lambda dxdt\) \\
 &=O\(\int_0^{R_\la}\frac{1}{t^3}\int_{B_t}(1+|x|^p)e^{4\eta_\la}dxdt\) =O(1),
 \end{split}\end{equation*}
 a contradiction to \eqref{limsupDeta2} and the definition of $R_\lambda$.
 \end{proof}
 
Using that $\eta_\lambda(0)=0$ and Lemma \ref{lem3}, together with ODE theory we get that, up to  a subsequence,
\begin{equation}\label{convetalambda}
\eta_\la\to\eta\quad \text{in } C^4_{loc}(\R^4),\text{ as }\lambda\to 0,
 \end{equation}
where the limit function $\eta$ satisfies
$$\D^2 \eta=(1-\mu |x|^p)e^{4\eta}\quad\text{in }\R^4,\quad \mu:=\lim_{\la\to0}\frac{r_\la^p}{\la}\in [0,\infty).$$ 
Notice that at this stage we do not know whether $\eta$ is a normal solution, $\mu>0$, and $\int_{\R^4}(1-|x|^p)e^{4\eta}dx=\Lambda$. This is what we are going to prove next.

\begin{lem}\label{lem-mu0} If $\mu=0$ then $e^{4\eta}\in L^1(\R^4)$. \end{lem} 
\begin{proof}   It follows  from Lemma \ref{lem1} and \eqref{convetalambda} that  $\Delta \eta$ is increasing, and $\lim_{r\to\infty}\D\eta(r)=:c_0\in[-\infty, 0]$. If $c_0<0$ then $\eta(r)\lesssim-r^2$, and hence $e^{\eta}\in L^1(\R^4)$. Therefore, if the lemma were false then necessarily we have $c_0=0$ and $e^{4\eta}\not\in L^1(\R^4)$. Then using \eqref{identity} one can show that for any $M>0$ large we have $\D\eta(r)\leq -\frac{M}{r^2}$ for $r\gg 1$. This in turn implies that $\eta(r)\leq -2\log r$ for $r\gg 1$, and hence $e^{4\eta}\in L^1(\R^4)$, a contradiction.\end{proof}

\begin{lem} \label{lem-mu} We have $\mu>0$.
\end{lem}
\begin{proof} Assume by contradiction that  $\mu=0$. Then $\eta$ is  a radial solution to
$$\D^2 \eta=e^{4\eta}\quad\text{in }\R^4,$$
with $e^{4\eta}\in L^1(\R^4)$ by Lemma \ref{lem-mu0}. By \cite[Theorem 2.1]{Lin}, either $\eta$ is spherical, namely $\eta(x)=\log\(\frac{2\lambda}{1+\lambda^2|x|^2}\)+\frac{\log 6}{4}$, for some $\lambda>0$, or there exists $c_0>0$ such that
\begin{equation}\label{Detac0}
-c_0:=\lim_{|x|\to\infty}\D\eta(x)<0.
\end{equation}
We shall now show that each of these two cases leads to a contradiction. \\
 
\noindent \textbf{Case 1:} \eqref{Detac0} holds.  

By Lemma \ref{lem1}, for every $\lambda>0$ we can find $0<R_{1,\la}< R_{2,\la}$ such that 
$$\eta_\la(R_{1,\la})=-\frac{c_0}{2},\quad \eta_\la(R_{2,\la})=-\frac{c_0}{4}.$$
Moreover \eqref{convetalambda} implies that $R_{1,\lambda},R_{2,\lambda}\to\infty$ as $\lambda\downarrow 0$.

Again by Lemma \ref{lem1} and \eqref{convetalambda} we have $\Delta \eta_\lambda\le - \frac{c_0}{4}$ in $[0,R_{2,\lambda}]$ and \eqref{identity} implies that $$\eta_\la(r)\leq -\frac{c_0}{32}r^2\quad\text{for }0\leq r\leq R_{2,\la}.$$
Applying \eqref{identity} with $w=\D \eta_\la$, and using \eqref{rlala}, we finally get
\begin{align*}
0< \frac{c_0}{4}&=\Delta \eta_\lambda(R_{2,\lambda})-\Delta \eta_\lambda(R_{1,\lambda})\\
&=\int_{R_{1,\la}}^{R_{2,\la}}\frac{1}{\omega_3 t^3}\int_{B_t}\(1-\frac{r_\lambda^p}{\la}|x|^p\)e^{-r_\la^2|x|^2}e^{4\eta_\la}dxdt\\
&=O\(\int_{R_{1,\la}}^{\infty}\frac{1}{t^3}\int_{B_t}(1+|x|^p)e^{-\frac{c_0}{8}|x|^2}dxdt \)\xrightarrow{\la\to0}0,
\end{align*}
which is a contradiction. \\

\noindent\textbf{Case 2:} $\eta$ is spherical, and in particular
$$\int_{\R^4}e^{4\eta}dx=\Ls.$$
Since $\Lambda<\Ls  $,   we can fix $R_0>0$ such that
$$\int_{B_{R_0}}e^{4\eta}dx>\Lambda.$$
Taking into account that $r_\lambda \to 0$ and by assumption $\frac{r_\lambda^p}{\lambda}\to 0$ as $\lambda\downarrow 0$, we can find $\lambda_0=\lambda_0(R_0)$ such that for
\begin{equation}\label{lem-muEq1}
\int_{B_{R_0}} \(1-\frac{r_\la^p}{\la}|x|^p\)e^{-r_\la^2|x|^2}e^{4\eta_\la}dx\geq \Lambda\quad \text{for }0<\lambda<\lambda_0.
\end{equation}
Setting
$$\Gamma_\la(t):=\int_{B_t}\(1-\frac{r_\la^p}{\la}|x|^p\)e^{-r_\la^2|x|^2}e^{4\eta_\la}dx,$$ we see that $\Gamma_\la(0)=0$, $\Gamma_\la$ is monotone increasing on $[0,\frac{\la^{1/p}}{r_\la}]$, and then it decreases to $\Lambda$ on the interval $[\frac{\la^{1/p}}{r_\la},\infty)$. Together with \eqref{lem-muEq1} it follows that 
\begin{equation}\label{lem-mueq2}
\Gamma_\lambda(t)\geq \Lambda\quad\text{for }t\geq R_0, \;0<\lambda<\lambda_0.
\end{equation}
Applying \eqref{identity} with $R=\infty$, $w=\D\eta_\la$, and recalling that $\lim_{|x|\to\infty}\D\eta_\lambda(x)=0$,  we get for $r\geq R_0$
\begin{align*} \D\eta_\la(r)=-\int_r^\infty \frac{\Gamma(t)}{\omega_3 t^3}dt\leq -\frac{\Lambda}{2\omega_3} \frac{1}{r^2}=-(4+p+\delta)\frac{1}{2r^2},\end{align*}
where $\delta>0$ is such that $\Lambda_{*,p}+2\delta\pi^2=\Lambda$. Hence, for $t>R_0$,
$$\int_{B_t}\D\eta_\la dx\leq\int_{B_t\setminus B_{R_0}}\D\eta_\la dx \le -\frac{4+p+\delta}{4} \omega_3(t^2-R_0^2).$$
Again by \eqref{identity}, as $\eta_\la(R_0)=O(1)$ by \eqref{convetalambda}, we have
\begin{equation}\label{uniform0}
\begin{split}
\eta_\la(r)&\leq  \eta_\la(R_0)+\int_{R_0}^r \frac{-(4+p+\delta)(t^2-R_0^2)}{4t^3}dxdt\\
&=C(R_0) -\frac{4+p+\delta}{4} \log r,\quad r\geq R_0.
\end{split}
\end{equation}
This implies that
\begin{align}\label{uniform}
\lim_{R\to\infty}\lim_{\la\to0}\int_{B_R^c}(1+|x|^p)e^{4\eta_\la}dx=0.
\end{align}
It follows from \eqref{uniform} that as $\lambda\downarrow 0$,
$$\Lambda=\int_{\R^4}\(1-\frac{r_\la^p}{\la}|x|^p\)e^{-r_\la^2|x|^2}e^{4\eta_\la}dx\to \int_{\R^4}e^{4\eta}dx=\Ls,$$
a contradiction. 

This completes the proof of the lemma. 
\end{proof}
   
   \begin{rem} The above proof also works without using the fact that $e^{4\eta}\in L^1(\R^4)$. Indeed, trivially one can find $R_0>0$   as in Case 2, and proceed in a similar way.  \end{rem}
   
\medskip 

\noindent\emph{Proof of the existence part (completed).}    
From Lemma \ref{lem-mu}, choosing $R=\(\frac{4}{\mu}\)^{1/p}$ we obtain for $\lambda$ sufficiently small
$$1-\frac{r_\la^p}{\la} |x|^p\le 1-\frac{\mu}{2}|x|^p \le -\frac{\mu}{4}|x|^p,\quad |x|\ge R,$$
hence from \eqref{relation-4} we obtain
\begin{equation}\label{thm1eq1}
\begin{split}
\int_{B_R^c}\frac{\mu}{4}|x|^p e^{-r_\la^2|x|^2}e^{4\eta_\la}dx &\le -\int_{B_R^c} \(1-\frac{r_\la^p}{\la} |x|^p\) e^{-r_\la^2|x|^2}e^{4\eta_\la}dx\\
&= \int_{B_R} \(1-\frac{r_\la^p}{\la} |x|^p\) e^{-r_\la^2|x|^2}e^{4\eta_\la}dx-\Lambda \le C,
\end{split}
\end{equation}
where in the last inequality we also used that $\eta_\lambda\le 0$. Since, the integrand in $B_R$ is uniformly bounded, it follows at once that
\begin{align}\label{volume}
\int_{\R^4}(1+ |x|^p)e^{-r_\la^2|x|^2}e^{4\eta_\la}dx\leq C.
\end{align}
Moreover \eqref{thm1eq1} also implies
$$\int_{B_R^c} e^{-r_\la^2|x|^2}e^{4\eta_\la}dx \le \frac{C}{R^p} \to 0,\quad \text{as }R\to\infty,$$  
uniformly with respect to $\lambda$, which in turn yields
\begin{equation}\label{thm1eq2}
 \lim_{\lambda\to0} \int_{\R^4}e^{-r_\la^2|x|^2}e^{4\eta_\la}dx=\int_{\R^4}e^{4\eta}dx.
 \end{equation}
By Fatou's lemma
\begin{equation}\label{thm1eq3}
\int_{\R^4}|x|^pe^{4\eta}dx\leq \lim_{\lambda\to0}\int_{\R^4}|x|^pe^{-r_\la^2|x|^2}e^{4\eta_\la}dx.
\end{equation}
Then \eqref{thm1eq2} and \eqref{thm1eq3} give
\begin{align}\label{int-limit}
\int_{\R^4}(1-\mu |x|^p)e^{4\eta}dx\geq\Lambda.
\end{align}

We now proceed as in Case 2 of the proof of Lemma \ref{lem-mu} to show that the above inequality is actually an equality and that  $\eta$ is a normal solution.  Since
$$\frac{\lambda^\frac1p}{r_\la}\to\frac{1}{\mu^\frac1p}>0,$$
we have for $R_0 =2\mu^{-\frac1p}$ and $\lambda_0=\lambda_0(R_0)$ sufficiently small that \eqref{lem-mueq2} holds, hence, as before \eqref{uniform0}-\eqref{uniform} follow. In particular \eqref{uniform} implies that
$$\int_{\R^4}(1-\mu |x|^p)e^{4\eta}dx=\Lambda,$$
and by taking the limit using \eqref{convetalambda} and \eqref{uniform0} we obtain  
\begin{equation}\label{etaint}
\begin{split}
\eta(x)\leftarrow \eta_\lambda(x)&= \frac{1}{8\pi^2}\int_{\R^4}\log\left(\frac{1}{|x-y|}\right)\(1-\frac{r_\la^p}{\la} |y|^p\)e^{-r_\la^2|y|^ 2}e^{4\eta_\la(y)}dy+c_\la\\
&\to \frac{1}{8\pi^2}\int_{\R^4}\log\left(\frac{1}{|x-y|}\right)\(1-\mu |y|^p\)e^{4\eta(y)}dy+c,
\end{split}
\end{equation}
where the identity in the first line follows from \eqref{prop1eq1} and \eqref{defetalambda}. In particular we have shown that $\eta$ is a normal solution.

We now set
$$u(x)=\eta(\rho x)+\log\rho,\quad \rho:=\mu^{-\frac{1}{4p}},$$
and with a simple change of variable we get
\begin{equation}\label{uint}
u(x)=\frac{1}{8\pi^2}\int_{\R^4}\log\left(\frac{1}{|x-y|}\right)\(1-|y|^p\)e^{4u(y)}dy+c
\end{equation}
so that $u$ is a normal solution to \eqref{eq-0}.
\endproof
 
\subsection{Asymptotic behaviour}

\noindent\emph{Proof of \eqref{asymp}} 
Consider the Kelvin transform of $u$ given by \eqref{defKelvin}.
By Proposition \ref{PKelvin} $\tilde u$ satisfies
\begin{equation}\label{Kelvin3}
\begin{split}
\tilde u(x)&=\frac{1}{8\pi^2}\int_{\R^4}\log\(\frac{1}{|x-y|}\)\(1-\frac{1}{|y|^p}\)\frac{e^{4\tilde u(y)}}{|y|^{8-\Lambda/2\pi^2}}dy+c.
\end{split}
\end{equation}
In particular
$$\Delta^2 \tilde u(x)= \(1-\frac{1}{|x|^p}\)\frac{e^{4\tilde u(x)}}{|x|^{8-\Lambda/2\pi^2}}=O\(\frac{1}{|x|^{8+p-\Lambda/2\pi^2}}\),\quad \text{as }|x|\to 0.$$
Observing that for $\Lambda>\Lambda_{*,p}$ we have
$$8+p-\frac{\Lambda}{2\pi^2}=4-\frac{\Lambda-\Lambda_{*,p}}{2\pi^2}<4,$$
we get
\begin{equation}\label{Delta^2uLp}
\Delta^2 \tilde u \in L^q_{loc}(\R^4) \quad \text{for }1\le q<\frac{1}{1-\frac{\Lambda-\Lambda_{*,p}}{8\pi^2}},
\end{equation}
hence by elliptic estimates $\tilde u\in W^{4,q}_{loc}(\R^4)$ with $p$ as in \eqref{Delta^2uLp}, and by the Morrey-Sobolev embedding $\tilde u\in C^{0,\alpha}_{loc}(\R^4)$ for $\alpha\in [0,1]$ such that $\alpha<\frac{\Lambda-\Lambda_{*,p}}{2\pi^2}$. 
Then \eqref{asymp} follows.
Alternatively to the elliptic estimates, the same $C^\alpha$ regularity can also be obtained directly from \eqref{Kelvin3}, using the H\"older inequality and the following estimate: For any $r>0$
\begin{align*}
\int_{B_1}\left| \log|z-h|-\log|z| \right|^r dz\leq C(r)\left\{
\begin{array}{ll}|h|^r&\quad\text{for }r< 4\\
|h|^r|\log|h||&\quad\text{for }r= 4\\
|h|^4&\quad\text{for }r>4,
\end{array}\right.
\end{align*}
for $|h|>0$ small. 
\endproof

\noindent\emph{Proof of \eqref{asymp3}} For $\ell=1,2,3$ we differentiate in \eqref{uint} to get
\begin{equation*}
|\nabla^\ell u(x)|=O\(\int_{\R^4}\frac{1}{|x-y|^\ell} (1+|y|^p)e^{4u(y)}dy\)
\end{equation*}
Since $\Lambda>\Lambda_{*,p}$, by \eqref{asymp}  we have that $$(1+|x|^p)e^{4u(x)}\leq \frac{C}{1+|x|^{4+\delta}},$$ for some $\delta>0$. Therefore, for $|x|$ large
\begin{align*} |\nabla^\ell u(x)| &\leq C\left(   \int_{B_\frac{|x|}{2}}+\int_{B_{2|x|}\setminus B_\frac{|x|}{2}}+\int_{B_{2|x|} ^c}\right) \frac{1}{|x-y|^\ell}\frac{dy}{1+|y|^{4+\delta}} \\ &\leq \frac{C}{|x|^\ell} +\frac{C}{|x|^{4+\delta}}\int_{B_{2|x|}\setminus B_\frac{|x|}{2}}\frac{dy}{|x-y|^\ell}\\ &\leq \frac{C}{|x|^\ell}. \end{align*}
\endproof

\section{Proof of Theorem \ref{thm1b}}

Let $(u_k)$ be a sequence of  radial normal solutions to \eqref{eq-0} with $\Lambda=\Lambda_k\in [\Lambda_{*,p},\Ls)$, i.e.
\begin{equation}\label{eqIk}
u_k(x)=\frac{1}{8\pi^2}\int_{\R^4}\log\(\frac{|y|}{|x-y|}\) (1-|y|^p)e^{4u_k(y)}dy +c_k,
\end{equation}
and
\begin{equation}\label{volk}
\Lambda_k=\int_{\R^4}(1-|x|^p)e^{4u_k(x)}dx\to\bar\Lambda \in[\Lambda_{*,p},\Ls).
\end{equation}
We want to prove the following:

\begin{prop}\label{propu*} Up to a subsequence we have $u_k\to \bar u$ uniformly locally in $\R^4$  where $\bar u$ is a normal solution to \eqref{eq-0} with $\Lambda=\bar \Lambda$
\end{prop}

In the following we shall use several times that $u_k$ is radially decreasing. This follows with the same proof of Lemma \ref{lem1}.

\begin{lem}\label{ukgeqC} We have $u_k (0)\geq -C$ where $C$ only depends on $\inf_k \Lambda_k$.
\end{lem}

\begin{proof}
We have
\begin{equation*}
\Lambda_k = \int_{\R^4} (1-|x|^p)e^{4u_k}dx\le \int_{B_1}e^{4u_k(x)}dx\le |B_1|e^{4u_k(0)},
\end{equation*}
where in the last inequality we used that $u_k$ is monotone decreasing.
\end{proof}

Since $\Lambda_k\in [\La_*,\Ls)$, we have the following Pohozaev identity (see Proposition \ref{poho-3}, which can be applied thanks to Lemma \ref{lemasymp} if $\Lambda\in (\Lambda_{*,p},\Ls)$ and thanks to Lemma \ref{lem-4.6} if $\Lambda=\Lambda_{*,p}$):  \begin{align}\label{poho}
 \frac{\Lambda_k}{\Ls}(\La_k-\Ls)=-\frac{p}{4}\int_{\R^4} |x|^pe^{4u_k}dx.
 \end{align}
 Therefore, by \eqref{eq-0} we get that
 \begin{align}\label{V1}
 \int_{\R^4}e^{4u_k}dx= \La_k+\frac{4\La_k}{p\Ls}(\Ls-\La_k).
 \end{align}
This yields
\begin{align}\label{V2}
\lim_{k\to\infty} \int_{\R^4}e^{4u_k}dx = \bar \La+\frac{4\bar\La}{p\Ls}(\Ls-\bar\La).
\end{align}

\begin{lem}\label{limsupuk}
We have
$$\limsup_{k\to\infty} u_k(0)<\infty.$$
\end{lem} 
\begin{proof} It follows from {\eqref{volk}}, \eqref{poho} and \eqref{V1} that
\begin{align}\label{unibound}
\limsup_{k\to \infty}\int_{\R^4}(1+|x|^p)e^{4u_k}dx<\infty.
\end{align}
Then, differentiating in \eqref{eqIk}, integrating over $B_1$ and using Fubini's theorem and \eqref{unibound}, one obtains
\begin{equation}\label{nablauk}
\int_{B_1}|\nabla u_k|dx\leq C \int_{\R^4}\(\int_{B_1}\frac{1}{|x-y|}dx \)(1+|y|^p)e^{4u_k(y)}dy \le C.
\end{equation}
Hence, if (up to a subsequence) $u_k(0)\to\infty$ as $k\to\infty$, by \cite[Theorem 2]{MarOpen} (see also \cite{HydCV} and \cite{Rob}) the blow-up at the origin is spherical, i.e.
$$u_k(r_kx)-u_k(0)+\log(2)=: \eta_k(x)\to \log\frac{2}{1+|x|^2},\quad \text{locally uniformly},$$
where $r_k:= 12e^{-u_k(0)}\to 0$ as $k\to\infty$,
and, we have quantization of mass in the sense that
$$\lim_{k\to\infty}\int_{B_\frac12}(1-|x|^p)e^{4u_k}dx= \Ls.$$
As $u_k$ is monotone decreasing, we  have  that  $u_k\to-\infty$ locally uniformly in $\R^4\setminus\{0\}$. Consequently, using \eqref{unibound} we get
\begin{align}\label{V2b}
\lim_{k\to\infty}\int_{\R^4}e^{4u_k}dx =\Ls.
\end{align}
On the other hand, comparing \eqref{V2b} with \eqref{V2}, and recalling that $\Lambda_{*,p}\le \bar\Lambda<\Ls$ and $\La_{*,p}=\frac18(4+p)\Ls$, we obtain
$$1=\frac{4\bar\Lambda}{p\Ls}\ge \frac{4\Lambda_{*,p}}{p\Ls}=\frac{4+p}{2p}>1$$
for $p\in (0,4)$, a contradiction.
\end{proof}

\begin{lem}\label{lemconv}
We have $u_k\to \bar u$, where $\bar u$ is a normal solution to \eqref{eq-0} for some $\Lambda=\tilde \Lambda \ge \bar\Lambda$.
\end{lem}

\begin{proof}  Since $u_k\leq u_k(0)=O(1)$ by Lemma \ref{ukgeqC} and Lemma \ref{limsupuk},  we have
$$\D^2 u_k=O_R(1)\quad \text{on }B_R.$$
Differentiating under the integral sign in \eqref{eqIk}, integrating over $B_R$ and using Fubini's theorem, together with \eqref{unibound}, we get
$$\int_{B_R}|\D u_k|dx\leq CR^2,\quad\text{for every }R>1.$$
Hence, by elliptic estimate, up to a subsequence,  $u_k\to \bar u$ in $C^{3}_{loc}(\R^4)$. 


To prove that $\bar u $ is normal, first note that the constant $c_k=u_k(0)$ in \eqref{eqIk}. Moreover For a fixed $x\in \R^4$ we have as $R\to\infty$  
$$\int_{B_R^c}\log\left(\frac{|y|}{|x-y|}\right)(1-|y|^p)e^{4u_k(y)}dy=O\(\frac{|x|}{R}\)\int_{B_R^c} (1+|y|^p)e^{4u_k(y)}dy=O\(\frac{|x|}{R}\),$$
thanks to \eqref{unibound}.  Therefore, using the convergence $u_k\to\bar u$ in $C^0_{loc}(\R^4)$, we conclude from \eqref{eqIk} that 
\begin{equation*}
\begin{split}
\bar u(x)& \xleftarrow{k\to\infty} u_k(x) =\frac{1}{8\pi^2}\int_{B_R}\log\left(\frac{|y|}{|x-y|}\right)(1-|y|^p)e^{4 u_k(y)}dy+O\(\frac{|x|}{R}\)+ u_k(0)\\
&\xrightarrow{k\to\infty} \frac{1}{8\pi^2}\int_{B_R}\log\left(\frac{|y|}{|x-y|}\right)(1-|y|^p)e^{4\bar u(y)}dy+O\(\frac{|x|}{R}\)+ \bar u(0)\\
&\xrightarrow{R\to\infty} \frac{1}{8\pi^2}\int_{\R^4}\log\left(\frac{|y|}{|x-y|}\right)(1-|y|^p)e^{4\bar u(y)}dy+ \bar u(0),
\end{split}
\end{equation*}
where in the last line we used dominated convergence, which is possible since $\log\left(\frac{|y|}{|x-y|}\right)=O(1)$ as $|y|\to\infty$, and $(1+|\cdot|^p)e^{4\bar u}\in L^1(\R^4)$ by \eqref{unibound} and  Fatou's lemma.

Still by Fatou's lemma, we also get
$$\tilde\Lambda:=\int_{\R^4}(1-|x|^p)e^{4\bar u}dx\geq \lim_{k\to\infty}\int_{\R^4}(1-|x|^p)e^{4 u_k}dx = \bar \La.$$
\end{proof}

\begin{lem}\label{lemLambdabar}
We have $\tilde\Lambda=\bar \Lambda$.
\end{lem}

\begin{proof} We assume by contradiction that $\tilde\Lambda>\bar \La$.  Since $u_k\to \bar u$ in $C^0_{loc}(\R^4)$, this is equivalent to
\begin{equation}\label{defrho}
\rho:=\lim_{R\to\infty}\lim_{k\to\infty}\int_{B_R^c}(|x|^p-1)e^{4u_k}dx=\lim_{R\to\infty}\lim_{k\to\infty}\int_{B_R^c}|x|^pe^{4u_k}dx=\tilde\La-\bar \La>0.
\end{equation}
We consider the Kelvin transform
$$\tilde u_k(x)=u_k\(\frac{x}{|x|^2}\)-\frac{\La_k}{8\pi^2} \log |x|, \quad x\ne 0.$$
{By Proposition \ref{PKelvin} we have
$$\tilde u_k(x)=\frac{1}{8\pi^2}\int_{B_1}\log\left(\frac{1}{|x-y|}\right)\(1-\frac{1}{|y|^{p}}\)\frac{e^{4\tilde u_k(y)}}{|y|^{4-p-\delta_k}}dy,\quad \delta_k:=\frac{\La_k-\La_{*,p}}{2\pi^2}.$$
In fact, with the same proof of \eqref{36} we obtain}
\begin{equation}\label{37}
\tilde u_k(x)=-\frac{1}{8\pi^2}\int_{B_1}\log\left(\frac{1}{|x-y|}\right)\frac{e^{4\tilde u_k(y)}}{|y|^{4-\delta_k}}dy+O(1)\quad\text{for } x\in B_1.
\end{equation}
If $\delta_k\not\to 0$ then from \eqref{37} we easily see that $\tilde u_k=O(1)$ in $B_1$, a contradiction to our assumption that $\rho>0$. Let us then assume that $\delta_k\to 0$, i.e. $\Lambda_k\to \Lambda_{*,p}$, and let $\ve_k>0$ be such that
$$\int_{B_{\ve_k}}\frac{e^{4\tilde u_k(y)}}{|y|^{4-\delta_k}}dy=\frac\rho2.$$
Then clearly $\ve_k\to 0$ as $k\to\infty$. Using that $\log\(\frac{1}{|x-y|}\)=\log\(\frac{1}{|x|}\)+O(1)$ for $|y|\leq\ve_k$, $ |x|\geq 2\ve_k$, and that $\log\(\frac{1}{|x-y|}\)$ is lower bounded for $y\in B_1$ and $x\to 0$, we get
\begin{equation}\label{eq64}
\begin{split}
\tilde u_k(x)&=-\frac{\log(1/|x|)}{8\pi^2}\int_{B_{\ve_k}}\frac{e^{4\tilde u_k(y)}}{|y|^{4-\delta_k}}dy -\frac{1}{8\pi^2} \int_{B_1\setminus B_{\ve_k}} \log\(\frac{1}{|x-y|}\)\frac{e^{4\tilde u_k(y)}}{|y|^{4-\delta_k}}dy +O(1)\\
& \leq-\frac{\rho}{16\pi^2}\log\(\frac{1}{|x|}\)+C\quad\text{for }2\ve_k\leq |x|\leq 1,
\end{split}
\end{equation}
which, in particular implies
\begin{equation}\label{uLambda0}
\lim_{r\to0}\lim_{k\to\infty}\sup_{B_r}\tilde u_k=-\infty.
\end{equation}
From \eqref{eq64} we immediately infer
$$\lim_{r\to0}\lim_{k\to\infty}\int_{B_r\setminus B_{2\ve_k}}\frac{e^{4\tilde u_k(y)}}{|y|^{4-\delta_k}}dy=0,$$
hence, also recalling \eqref{defrho},
$$\lim_{k\to\infty}\int_{B_{2\ve_k}}\frac{e^{4\tilde u_k(y)}}{|y|^{4-\delta_k}}dy=\rho.$$
This, and using \eqref{uLambda0} we get
$$\frac\rho2= \lim_{k\to\infty}\int_{B_{2\ve_k}\setminus B_{\ve_k}}\frac{e^{4\tilde u_k(y)}}{|y|^{4-\delta_k}}dy=o(1)\int_{B_{2\ve_k}\setminus B_{\ve_k}}\frac{dy}{|y|^4}=o(1),\quad \text{as }k\to\infty,$$
contradiction.
\end{proof}

\noindent\emph{Proof of Theorem \ref{thm1b}.} With Lemma \ref{lemconv} and Lemma \ref{lemLambdabar} we have the desired convergence (up to a subsequece) of $u_k$ to $\bar u$, a normal solution of \eqref{eq-0} with $\Lambda=\bar\Lambda$. The asymptotic behaviour \eqref{asymp2} follows from Lemma \ref{lem-4.6}, while for \eqref{asymp3}, the same proof used for $\Lambda\in (\Lambda_{*,p},\Ls)$ also works for the case $\Lambda=\Lambda_{*,p}$.
\endproof

\section{Proof of Theorem \ref{thm2}}

\begin{lem}\label{ukbdd} Let $(u_k)$ be a sequence solving \eqref{eq-0} with $\Lambda=\Lambda_k\uparrow \Ls$. Then we have $u_k(0)\to\infty$ as $k\to\infty$.
\end{lem}

\begin{proof} By Lemma \ref{ukgeqC} we have $u_k(0)\ge -C$. Assume by contradiction that, up to a subsequence, $u_k(0)\to  \ell \in \R$. Then, by Lemma \ref{lemconv} (or, rather, following its proof) we have $u_k\to \bar u$, normal solution to \eqref{eq-0} for some $\Lambda\ge \Ls$, contradicting Theorem \ref{thm0}.
\end{proof}

Differentiating under the integral sign in \eqref{eqIk} and integrating over $B_1$ we obtain \eqref{nablauk}. By Lemma \ref{ukbdd}, the sequence $(u_k)$ blows up at the origin. This, \eqref{nablauk} and \cite[Theorem 2]{MarOpen} imply \eqref{convetak}.
This completes the proof of Theorem \ref{thm2}. \hfill$\square$

\section{Proof of Theorem \ref{thm-positive2}}

We start by looking for normal solutions with prescribed value at the origin.

\begin{thm} \label{thm-positive} For every $p>0$ and $\rho\in \R$ there exists a unique radially symmetric normal solution to
\begin{align}\label{eq-rho}
\D^2u=(1+|x|^p)e^{4u}\quad\text{in }\R^4,\quad u(0)=\rho,\quad (1+|x|^p)e^{4u}\in L^1(\R^4). \end{align}
\end{thm}

\begin{proof}
For every $\ve>0$ we claim that there exists a radial normal solution to
\begin{equation}\label{veps}
\D^2 v_\ve=(1+|x|^p)e^{-\ve |x|^2}e^{4v_\ve}\quad\text{in }\R^4,\quad v_\ve(0)=\rho.
\end{equation}
To this end, we set
$$X=\{v\in C^0_{rad}(\R^4):\|v\|_X<\infty\},\quad \|v\|_X:=\sup_{x\in\R^4}\frac{|v(x)|}{\log(2+|x|)},$$
and define the operator $T_\ve:X\to X$, $T_\ve v=\bar v$, where $$\bar v(x):=\frac{1}{8\pi^2}\int_{\R^4}\log\left(\frac{|y|}{|x-y|}\right)(1+|y|^p)e^{-\ve|y|^2}e^{4v(y)}dy+\rho.$$
By the Arzerl\`a-Ascoli theorem it follows that $T_\ve$ is compact. 

Notice that $\Delta \bar v<0$, hence $\bar v$ is monotone decreasing by \eqref{identity}, and
$$\bar v\leq \bar v(0)=\rho.$$
In particular, if $v$ is a solution to $v=tT_\ve (v)$ with $0<t\leq 1$, then $$|v(x)|\leq  t\rho+\frac{te^{4\rho}}{8\pi^2}\int_{\R^4}\left|\log\left(\frac{|y|}{|x-y|}\right)\right|(1+|y|^p)e^{-\ve|y|^2} dy\leq C\log(2+|x|),$$ where $C>0$ is independent of $v$ and $t$. Then, by the Schauder fixed-point theorem, $T_\ve$ has a fixed point, which we call $v_\ve$, and which is a radial normal solution to \eqref{veps} by definition, so the claim is proven.
 
Setting
$$\Lambda_\ve=\int_{\R^4}(1+|x|^p)e^{-\ve|x|^2}e^{4 v_\ve(x)}dx,$$
and using the Pohozaev identity (Proposition A.1 in \cite{HMM} or a minor modification of  Proposition \ref{poho-3}, with $K(x)=(1+|x|^p)e^{-\ve|x|^2}$), we get
 $$\frac{\Lambda_\ve}{\Ls}(\Lambda_\ve-\Ls)<\frac{p}{4}\int_{\R^4}|x|^pe^{-\ve|x|^2}e^{4 v_\ve(x)}dx<\frac p4\Lambda_\ve,$$
 which implies
 $$\Lambda_\ve<\(1+\frac p4\)\Ls.$$
 Now one can follow the proof of Lemma \ref{lemconv} to conclude that $v_\ve\to v$, where $v$ is  a normal solution with $v(0)=\rho$. 
 
The uniqueness follows by the monotonicity property of solutions to ODEs with respect to the initial data.  
\end{proof}

Notice that the result of Theorem \ref{thm-positive} does not hold in the case of Problem \eqref{eq-0}, see Lemma \ref{ukgeqC}.

\begin{lem}\label{T6L1} Let $u$ be a normal solution to \eqref{eq-positive} for some $\Lambda>0$. Then we have
\begin{equation}\label{u+low}
u(x)\geq -\frac{\Lambda}{8\pi^2}\log|x|+O(1)\quad\text{as }|x|\to\infty,
\end{equation}
hence $\Lambda>\Lambda_{*,p}$.
\end{lem}

\begin{proof} 
The proof of \eqref{u+low} follows as in Lemma \ref{lemasymp} by changing the sign of $K$. Using that $|\cdot|^pe^{4u}\in L^1(\R^4)$ together with \eqref{u+low} we then infer that $\Lambda>(1+\frac p4)\frac{\Ls}{2}=\Lambda_{*,p}$.
\end{proof} 

\begin{lem}\label{T6L2} Let $u$ be a normal solution to \eqref{eq-positive} for some $\Lambda>0$. Then we have
\begin{equation}\label{u+upp}
u(x)\le -\frac{\Lambda}{8\pi^2}\log|x|+o(\log|x|)\quad\text{as }|x|\to\infty,
\end{equation}
\end{lem}

\begin{proof}
From the same proof of \cite[p. 213]{Lin}, we easily get that for every $\ve>0$ there exists $R(\ve)>0$ such that 
\begin{equation}\label{estimateLin}
u(x)\le \(-\frac{\Lambda}{8\pi^2}+\ve\)\log|x| +\frac{1}{8\pi^2} \int_{B_1(x)}\log \(\frac{1}{|x-y|}\) K(y)e^{4u(y)}dy, \quad |x|\ge R(\ve),
\end{equation}
where $K(y)=1+|y|^p$.
As in \cite[Lemma 3.5]{HMM}, from \eqref{estimateLin} and Jensen's inequality we get that for every $\ve'>0$ and $q\ge 1$ there is a constant $C=C(\ve',q)$ such that
$$\int_{B_1(x)}e^{4qu}dy \le C|x|^{-\(\frac{\Lambda}{2\pi^2}-\ve'\)q}.$$
With H\"older's inequality we then infer for $|x|$ large
\begin{equation}\label{T6L2E2}
\begin{split}
\int_{B_1(x)}\log \(\frac{1}{|x-y|}\) K(y)e^{4u(y)}dy&\le C|x|^p \int_{B_1(x)}\log \(\frac{1}{|x-y|}\) e^{4u(y)}dy\\
&\le C|x|^p\|e^{4u}\|_{L^q(B_1(x))}\\
&\le C|x|^{p-\frac{\Lambda}{2\pi^2}+\ve'}\le C
\end{split}
\end{equation}
since by Lemma \ref{T6L1} we have $\Lambda> \Lambda_{*,p}>\frac{p\Ls}{8}=2\pi^2 p$ and we can choose $0<\ve'< \frac{\Lambda}{2\pi^2}-p$.
Plugging \eqref{T6L2E2} into \eqref{estimateLin}, we obtain \eqref{u+upp}.
\end{proof}

\begin{lem}\label{T6L3} Let $u$ be a normal solution to \eqref{eq-positive} for some $\Lambda\in \R$. Then $u$ satisfied the Pohozaev identity \eqref{general Poho} and
\begin{equation}\label{neceb}
\max\{\Ls,\Lambda_{*,p}\}<\Lambda<2\Lambda_{*,p}.
\end{equation}
\end{lem} 
\begin{proof} From Lemma \ref{T6L1} we have $\Lambda>\Lambda_{*,p}$, which together with Lemma \ref{T6L2} implies that \eqref{cond-poho} is satisfied, hence the Pohozaev identity \eqref{general Poho} holds. From it we obtain
\begin{equation}\label{pohopositive}
\frac{\Lambda}{\Ls}(\Lambda-\Ls)=\frac{p}{4}\int_{\R^4}|x|^pe^{4u}dx.
\end{equation}
Since
$$0< \int_{\R^4}|x|^pe^{4u}dx <\Lambda,$$
\eqref{pohopositive} implies
$$\Ls<\Lambda<\(1+\frac{p}{4}\)\Ls=2\Lambda_{*,p},$$
and since we have already proven that $\Lambda>\Lambda_{*,p}$, \eqref{neceb} follows.
\end{proof}

\begin{lem}\label{ukcontinuous} Let $(u_k)$ be a sequence of radially symmetric normal solutions to \eqref{eq-positive} with $\Lambda=\Lambda_k$. If the sequence  $(u_k(0))$ is bounded, then also the sequence $(\Lambda_k)$ is bounded and, up to a subsequence, $u_k\to \bar u$, $\Lambda_k\to\bar\Lambda\in (0,\infty)$, where $\bar u$ is a normal solution to \eqref{eq-positive} with $\Lambda=\bar\Lambda.$
 \end{lem}
 
\begin{proof} By Lemma \ref{T6L3} the Pohozaev identity \eqref{general Poho} holds, hence
$$\frac{\Lambda_k}{\Ls}\left(\Lambda_k-\Ls\right)=\frac p4\int_{\R^4}|x|^pe^{4u_k}dx<\frac p4 \Lambda_k,$$
which then implies
$$\Lambda_k<\(1+\frac p4\)\Ls=2\Lambda_{*,p}.$$
Following the proof of Lemma \ref{lemconv} (with inequality reversed in the last line because of the positivity of $K$) we obtain that, up to a subsequence $u_k\to \bar u$, normal solution to \eqref{eq-positive} for some $\Lambda =\tilde \Lambda\le \bar \Lambda$.

To prove that $\tilde \Lambda =\bar \Lambda$ it suffices to show that
\begin{equation}\label{49}
\lim_{R\to\infty}\lim_{k\to\infty}\int_{B_R^c}|x|^pe^{4u_k}dx=0.
\end{equation}
Upon the Kelvin transform
$$\tilde u_k(x):=u_k\(\frac {x}{|x|^2}\)-\frac{\Lambda_k}{8\pi^2}\log|x|,$$
\eqref{49} is equivalenti to
\begin{align}\label{50}
\lim_{r\to0}\lim_{k\to\infty}\int_{B_r}|x|^{p_k}e^{4\tilde u_k}dx=0,\quad p_k:=\frac{\Lambda_k}{2\pi^2}-p-8.
\end{align}
By Proposition \ref{PKelvin}, $\tilde u_k$ is a normal solution to
$$\Delta^2 \tilde u_k=\(1+|x|^p\)|x|^{p_k}e^{4\tilde u_k},\quad \int_{\R^4}\(1+|x|^p\)|x|^{p_k}e^{4\tilde u_k}dx=\Lambda_k.$$
Since $\bar u$ is a normal solution, by Lemma \ref{T6L1} we have $\tilde \Lambda>\Lambda_{*,p}$. Therefore $\Lambda_{*,p}<\tilde \Lambda\leq \bar \Lambda$, which implies that for some $\delta>0$ and $k$ large, we have $p_k\geq -4+\delta$. Hence, \eqref{50} will follow if we show that 
\begin{equation}\label{77}
\lim_{k\to\infty}\sup_{B_1}\tilde u_k<\infty.
\end{equation}
First we note that by differentiating the integral formula of $\tilde u_k$ we obtain $\Delta u_k<0$, hence by \eqref{identity} we have that $\tilde u_k$ is monotone decreasing. Therefore, if $\tilde u_k(0)\to\infty$, then up to a subsequence, the   rescaled function $$\eta_k(x)=\tilde u_k(r_kx)-\tilde u_k(0),\quad r_k:= e^{-\frac{4}{4+p_k}\tilde u_k(0)},$$
which is a normal solution to
$$\Delta^2 \eta_k=(1+o(1))|x|^{p_k}e^{4\eta_k},$$
with $o(1)\to 0$ locally uniformly, converges to a limit function $\eta$ (see the proof of Proposition 4.1 in \cite{HMM} for details),  where $\eta$ is a normal solution to
$$\D^2 \eta=|x|^{p_\infty}e^{4\eta},\quad p_\infty:=\lim_{k\to\infty}p_k>-4.$$   Then by  \cite[Theorem 1]{HMM} we have
$$ \int_{\R^4}|x|^{p_\infty}e^{4\eta} dx=\(1+\frac{p_\infty}4\)\Ls.$$
Then, from Fatou's lemma we have
\begin{equation}\label{78}
\lim_{r\to0}\lim_{k\to\infty}\int_{B_r}|y|^{p_k}e^{4\tilde u_k}dy\geq  \(1+\frac{p_\infty}4\)\Ls.
\end{equation}
Moreover, as in the proof of Lemma \ref{lemLambdabar} we estimate
\begin{align*}
\tilde u_k(x)&\geq\frac{1}{8\pi^2}\int_{B_1}\log\(\frac{1}{|x-y|}\)|y|^{p_k}e^{4\tilde u_k(y)}dy+O(1),\quad x\in B_1.
\end{align*}
Since $\tilde u_k\to\tilde{u}$ outside the origin, where $\tilde u=\tilde{\bar u}$ is the Kelvin transform of $\bar u$, together with \eqref{78}
we get that
\begin{align*}
\tilde{u}(x)&\geq \frac{1}{8\pi^2}\(1+\frac{p_\infty}{4}\)\Ls\log\(\frac{1}{|x|}\)+O(1)\\
&=2\(1+\frac{p_\infty}{4}\)\log\(\frac{1}{|x|}\)+O(1),\quad x\in B_1.
\end{align*}
In particular, as $p_\infty>-4$, we have $$|x|^{p_\infty}e^{4\tilde u(x)}\geq \delta |x|^{-4}\quad x\in B_1,$$ for some $\delta>0$. This shows that  $|x|^{p_\infty}e^{4\tilde u(x)}\not\in L^1(B_1)$, however, by Fatou's lemma $$\int_{B_1}|x|^{p_\infty}e^{4\tilde u(x)}dx\leq \liminf_{k\to\infty }\int_{B_1}|x|^{p_k}e^{4\tilde u_k(x)}\leq \liminf_{k\to\infty}\Lambda_k\leq 2\Lambda_{*,p}.$$
This contradiction completes the proof of \eqref{77}, hence of the lemma.
\end{proof}

\noindent\emph{Proof of Theorem \ref{thm-positive2} (completed).} We have already proven the necessary conditions \eqref{nece0}-\eqref{nece} in Lemma \ref{T6L3}, so it remains to prove the existence part and the necessary condition \eqref{nece2} in the radial case with $p>4$.

By Lemma \ref{ukcontinuous}, the map
$$\R\ni\rho\mapsto \Lambda_\rho:=\int_{\R^4}(1+|x|^p)e^{4u_\rho}dx$$
is continuous, where $u_\rho$ is the solution to \eqref{eq-rho} given by Theorem \ref{thm-positive}.

We now have
\begin{equation}\label{T6E1}
\lim_{\rho\to -\infty}\Lambda_\rho= \(1+\frac p4\)\Ls=2\Lambda_{*,p},
\end{equation}
which is a consequence of \eqref{V1} and
$$\int_{\R^4}|x|^pe^{4u_\rho}dx\leq C\quad \Rightarrow\quad\int_{\R^4}e^{4u_\rho}dx\to0.$$ Taking $\rho\to\infty$ we see that the blow-up around the origin is spherical (see e.g. \cite{MarOpen}), and
$$\lim_{\rho\to+\infty}\int_{\R^4}e^{4u_\rho}dx= \Ls.$$
Again by \eqref{V1}, and as $\Lambda_\rho>\max\{\Ls,\Lambda_{*,p}\}$, we conclude that \begin{equation}\label{T6E2}
\lim_{\rho\to\infty }\Lambda_\rho=\max\left\{\Ls,\frac p4\Ls\right\}.
\end{equation}
Then, by continuity, we have existence for every $\max\{\Ls,\Lambda_{*,p}\}<\Lambda<2\Lambda_{*,p}$. 

\medskip

It remains to prove the stronger necessary condition \eqref{nece2} for $p>4$ in the radial case. 

Assume by contradiction that for a sequence $(\Lambda_k)$ with $\Lambda_k\downarrow \Lambda_{*,p}$ there are radial solutions $u_k$ to \eqref{eq-positive} with $\Lambda=\Lambda_k$. Since
$$\Lambda_{*,p}<\frac{p}{4}\Ls<2\Lambda_{*,p},$$
from \eqref{T6E1}-\eqref{T6E2} we obtain that the sequence $(u_k(0))$ is bounded. Then by Lemma \ref{ukcontinuous} we have that (up to a subsequence) $u_k\to \bar u$ locally uniformly, where $\bar u$ is a normal solution to \eqref{eq-positive} with $\Lambda=\Lambda_{*,p}$, and this contradicts Lemma \ref{T6L3}.
\endproof


\section{Proof of Proposition \ref{prop1}}\label{Sec:Prop2.1}

By \cite[Theorem 2.1]{CC} (and its proof), setting $K_\lambda= (\lambda-|x|^p)e^{-|x|^2}$ and given $\mu=1-\frac{\Lambda}{\Ls}\in (0,1)$ one can find a solution $u_\lambda$ to
$$\Delta^2 u_\lambda = K_\lambda e^{4u_\lambda}\quad \text{in }\R^4,$$
such that
$$\int_{\R^4}K_\lambda e^{4u_\lambda} dx=(1-\mu)\Ls=\Lambda.$$ 
Moreover $u_\lambda$ is of the form $u_\lambda= w\circ \Pi^{-1} +(1-\mu)\eta_0$ where $\eta_0(x)=\log\(\frac{2}{1+|x|^2}\)$, $\Pi:S^4\to\R^4$ denotes the stereographic projection, and $w\in H^2(S^4)$ minimizes a certain functional on $S^4$. This leads to the Euler-Lagrange equation
$$P^4_{g_0}w +6(1-\mu)= (K_\lambda \circ\Pi) e^{-4\mu (\eta_0\circ \Pi)} e^{4w},$$
where $P^4_{g_0}=\Delta_{g_0}(\Delta_{g_0}-2)$ is the Paneitz operator on $S^4$ with respect to the round metric $g_0$. Since $(K_\lambda \circ\Pi) e^{-4\mu (\eta_0\circ \Pi)}\in L^\infty(S^4)$, and $e^{4w}\in L^q(S^4)$ for every $q\in [1,\infty)$ by the Moser-Trudinger inequality, by elliptic estimates, we obtain $w\in C^{3,\alpha}(S^4)$ for $\alpha\in (0,1)$. In particular $w$ is continuous at the South pole $S=(0,0,0,0,-1)$ (the singularity of the stereographic projection), hence
$$u_\lambda(x) =  (1-\mu)\eta_0(x) + w(S) +o(1)= \frac{\Lambda}{8\pi^2}\log|x|+C +o(1)\quad \text{as }|x|\to\infty.$$
Now, setting
$$v_\lambda=\frac{1}{8\pi^2}\int_{\R^4}\log\(\frac{1}{|x-y|}\) K_\lambda(y)e^{4u_\lambda(y)}dy $$
we observe that $h_\lambda:=u_\lambda-v_\lambda$ satisfies
$$\Delta^2 h_\lambda=0,\quad h_\lambda(x)=O(\log|x|)\quad \text{as }|x|\to\infty.$$
Hence by the Liouville theorem we get $h_\lambda = const$. In particular $u_\lambda$ is a normal solution, i.e. it satisfies \eqref{prop1eq1}. 
This completes the proof of Proposition \ref{prop1}. \hfill$\square$

\end{document}